\documentclass[11pt]{amsart}
\usepackage{graphicx}
\usepackage{color}
\newtheorem{theorem}{Theorem}[section]
\newtheorem{corollary}[theorem]{Corollary}
\newtheorem{lemma}[theorem]{Lemma}
\newtheorem{proposition}[theorem]{Proposition}
\theoremstyle{definition}

\theoremstyle{remark}

\newtheorem{remark}[theorem]{Remark}
\numberwithin{equation}{section}
\begin{document}

\title[Stochastic Lie algebra]{On the stochastic Lie algebra}%
\author{Manuel Guerra, Andrey Sarychev}%
\address{ISEG and CEMAPRE, ULisboa, Rua do Quelhas 6, 1200-781 Lisboa, Portugal \\
University of Florence, DiMaI, v. delle Pandette 9, Firenze, 50127, Italy}%
\email{mguerra@iseg.ulisboa.pt,asarychev@unifi.it}%

\thanks{Manuel Guerra was partly supported by FCT/MEC through the project CEMAPRE – UID/MULTI/00491/2013.}%

\begin{abstract}
We study the structure of the Lie algebra $\mathfrak{s}(n,\mathbb R)$ corresponding to the so-called stochastic Lie group $\mathcal{S} (n,\mathbb R)$.
We obtain the Levi decomposition of the Lie algebra, classify  Levi factor and classify  the representation of the  factor in $\mathbb{R}^n$. We discuss isomorphism  of $\mathcal{S}(n,\mathbb R)$ with the group of invertible affine maps ${\it Aff}(n-1,\mathbb R)$.
We prove that $\mathfrak s(n, \mathbb R)$ is generated by two generic elements.
\end{abstract}
\maketitle

\section{Stochastic Lie group and stochastic Lie algebra}
\label{S introduction}

Let $\mathcal S_0^+(n,\mathbb R)$ denote the space of  \emph{transition matrices} of size $n$, i.e., the space of real $n \times n$ matrices with all entries non-negative and row sums equal to $1$.

One important motivation for the study of such matrices is their relation to Markov processes:
It is easy to see that for any Markov process $X$ with $n$ possible states, the family
\[
P(s,t) = \left[ p_{i,j}(s,t) \right]_{1 \leq i, j \leq n} ,
\qquad
0 \leq s \leq t < +\infty,
\]
where $p_{i,j}(s,t)$ is the probability of $X_t = j$, conditional on $X_s=i$, is a family of transition matrices such that
\begin{equation}
\label{Eq semigroup property}
P_{s,t} = P_{u,t}P_{s,u},
\qquad
\forall 0 \leq s \leq u \leq t < +\infty .
\end{equation}
Conversely, the Kolmogorov extension theorem (see e.g. \cite{Cinlar11}, Theorem 
\linebreak 
IV.4.18), states that for every family $\left\{ P(s,t) \in \mathcal S_0^+(n, \mathbb R) \right\}_{0 \leq s \leq t < +\infty}$ satisfying \eqref{Eq semigroup property}, there exists a Markov process $X$  such that $p_{i,j}(s,t) = \linebreak \Pr \left\{ \left. X_t=j \right| X_s=i \right\}$ for every $i,j \leq n$ and every $0 \leq s \leq t <+\infty$.

Let $\mathcal S^+(n , \mathbb R)$ denote the space of  \emph{nonsingular} transition matrices.
It is clear that $\mathcal S_0^+(n , \mathbb R)$ is a semigroup with respect to matrix multiplication, and $\mathcal S^+ (n , \mathbb R)$ is a subsemigroup.
However, $\mathcal S^+ (n , \mathbb R)$ is \emph{not} a group, since the inverse of a transition matrix is not, in general, a transition matrix.

The smallest group containing $\mathcal S^+(n, \mathbb R)$ is denoted by $\mathcal S(n, \mathbb R)$.
Due to the considerations above, this is called the \emph{stochastic group} \cite{Poole}.
It can be shown that
\[
\mathcal S(n,\mathbb R) = \left\{
P \in \mathbb R^{n \times n} : \mathrm{Det}(P) \neq 0, P \mathbf 1 = \mathbf 1
\right\},
\]
where $\mathbf 1$ is the $n$-dimensional vector with all entries equal to $1$.
It follows that $\mathcal S (n ,\mathbb R)$, provided with the topology inherited from the usual topology of $\mathbb R^{n \times n}$, is a $n\times (n-1)$ dimensional analytic Lie group.

The Lie algebra of $\mathcal S(n, \mathbb R)$ is called  \emph{stochastic Lie algebra}, and  is denoted by $\mathfrak{s} (n, \mathbb R)$.
Notice that $\mathfrak{s} (n, \mathbb R)$ is isomorphic to the tangent space of $\mathcal S(n, \mathbb R)$ at the identity
\[
\mathfrak{s} (n, \mathbb R) \sim
T_{Id}\mathcal S(n, \mathbb R) =
\left\{
A \in \mathbb R^{n \times n} : A \mathbf 1 = 0
\right\},
\]
$\mathfrak{s} (n, \mathbb R)$ is provided with the matrix commutator $[A,B] = AB - BA$.

We introduce the subset
\[
\mathfrak{s}^+ (n, \mathbb R) =
\left\{ A \in \mathfrak{s} (n, \mathbb R): a_{i,j} \geq 0, \ \forall i \neq j \right\}.
\]
It is clear that $\mathfrak{s}^+ (n, \mathbb R) $ is not a subalgebra of $\mathfrak{s} (n, \mathbb R) $, but it is a convex cone with nonempty interior in $\mathfrak{s} (n, \mathbb R)$.
Since
$\mathcal S^+ (n, \mathbb R)$ is invariant under the flow by ODE's of type
\[
\dot P_t = P_t A,
\]
with $A \in \mathfrak{s}^+(n, \mathbb R)$, it follows that $\mathcal S^+ (n, \mathbb R)$  has nonempty interior in $\mathcal S (n, \mathbb R)$.

In \cite{BoukasFeinsilverFellouris15}, it is stated that the Levi decomposition 
\begin{equation}\label{Eq splitting}
  \mathfrak{s} (n,\mathbb{R}) = {\mathfrak l} \oplus {\mathfrak r},
\end{equation}
has the following components:
\begin{itemize}
\item[a)]
The radical
$\mathfrak{r}$ is the linear subspace generated by the matrices
\begin{equation}\label{generators Boukas}
\hat R_i= E_i(n)-E_n(n), \quad i=1, \ldots , n-1, \qquad
\hat Z=Id -\frac{1}{n}J_n,
\end{equation}
where $E_i(n)$ are the matrices with the elements in the $i$-th column equal to $1$ and all other elements equal to zero, $J_n$ is the matrix with all elements equal to $1$;
\item[b)]
The Levi subalgebra
$\mathfrak{l}$ is the linear subspace of real traceless matrices with all row and column sums equal to zero.
\end{itemize}

The result is correct but the respective proof of \cite[Proposition 3.3]{BoukasFeinsilverFellouris15} seems to contain a logical gap in what regards the semisimplicity of ${\mathfrak l}$ and the maximality of  ${\mathfrak r}$.

In what  follows, we present an orthonormal basis for $\mathfrak{s}(n, \mathbb R)$ which has interesting properties with respect to the Lie algebraic structure of $\mathfrak{s}(n, \mathbb R)$.
In particular, it allows for the explicit computation of the Killing form and therefore we prove semisimplicity of $\mathfrak{l}$ by application of Cartan criterion.
We also obtain the Dynkin diagram of $\mathfrak l$, showing that it is isomorphic to $\mathfrak{sl}(n-1,\mathbb R)$.

\section{Basis for  the Lie algebra $\mathfrak s(n, \mathbb R)$}
\label{S Basis}

Choose an orthonormal
basis $v_1, \ldots , v_{n-1}$ of the hyperplane
\[
\Pi_n=\{x \in \mathbb{R}^n : \ x_1+ \ldots +x_n=0\} ,
\]
and set $v_0=\frac 1{\sqrt n} (1, \ldots , 1) \in \mathbb{R}^n$.
Recall that for $a, b \in \mathbb{R}^n$, the dyadic product $a \otimes b$ is the matrix:
\[
\left(
    \begin{array}{c}
      a_1 \\
      \vdots \\
      a_n \\
    \end{array}
  \right)  \otimes \left(
                    \begin{array}{ccc}
                      b_1 & \cdots  & b_n \\
                    \end{array}
                  \right)=\left(
                            \begin{array}{ccc}
                              a_1b_1 & \cdots & a_1b_n \\
                              \vdots & & \vdots \\
                              a_nb_1 & \cdots & a_nb_n \\
                            \end{array}
                          \right).
\]
The matrices
\begin{align}
& \label{Eq Def Z}
Z=\frac{1}{\sqrt{n-1}}\left(I_n-v_0 \otimes v_0\right) ,
\\ & \label{Eq Def Ri}
R_i=v_0 \otimes v_i, \qquad i=1, \ldots , n-1
\end{align}
span the same linear subspace as the matrices \eqref{generators Boukas}.

We take the $(n-1)(n-2)$-dimensional  linear subspace
\[
\mathcal A =
\mathrm{span} \left\{ A_{ij}, \ i=1, \ldots n-1, \ j=1, \ldots , n-1, \ i \neq j \right\} ,
\]
spanned by  the rank-$1$ matrices
\begin{equation}\label{Eq Def Aij}
A_{ij}= v_i \otimes v_j.
\end{equation}
Since  $v_i \in \Pi$, there holds $v_0^*(v_i \otimes v_j)=(v_0\cdot v_i)v_j^*=0$.
Similarly, 
\linebreak 
$(v_i \otimes v_j)v_0=0$.
Hence the matrices $A_{ij}$ have zero row and column sums.
Since ${\rm Tr}(v_i \otimes v_j)=v_i \cdot v_j=0$, the matrices $A_{ij}$ are traceless.

Now, consider the linear subspace
\begin{equation}\label{Eq Def space H}
\mathcal{H}=\left\{ H=\sum_{\ell =1}^{n-1}\gamma_{\ell}(v_\ell \otimes v_\ell)\left| \ \sum_{\ell=1}^{n-1}\gamma_\ell=0\right.\right\}.
\end{equation}
The row and column sums of each $(v_\ell \otimes v_\ell)$ are zero, and the trace of
$H \in \mathcal H$ equals $\sum\limits_{\ell =1}^{n-1}\gamma_{\ell}=0$.

We set
\begin{equation}\label{def_ell}
  \mathfrak l = \mathcal A \oplus \mathcal H .
\end{equation}

We introduce a basis of $\mathcal H$:
\begin{equation}
\label{Eq D Hk}
H_k= \sum\limits_{\ell =1}^{n-1}\gamma^k_{\ell}(v_\ell \otimes v_\ell), \qquad k=1, \ldots , (n-2),
\end{equation}
where $\gamma^k=(\gamma^k_1, \ldots , \gamma^k_{n-1}), \ k=1, \ldots , (n-2)$, form an orthonormal basis for the subspace
\[
\Pi_{n-1}= \left\{ x \in \mathbb R^{n-1} : x_1 + \ldots + x_{n-1}=0 \right\}.
\]
Using the definition of dyadic product and elementary properties of the trace, it is straightforward to check that the matrices
\begin{align*}
&
Z, \quad
R_i \ (i=1, \ldots, n-1), 
\\&
A_{ij} \  (i,j=1, \ldots , n-1, i \neq j), \quad
H_i \ (i=1, \ldots, n-2)
\end{align*}
form an orthonormal system with respect to the matrix scalar product
$\langle A, B \rangle = \mathrm{Tr}(AB^*)$.

The following Lemma presents the multiplication table for our basis. Its proof is accomplished by a  direct computation.

\begin{lemma}
\label{L multiplication table}
For meaningful values of the indexes $i,j,k,\ell$ there holds:
\begin{itemize}
\item[]
$[Z,R_i] = \frac{-1}{n-1}R_i $;
\item[]
$[Z,A_{ij}]=0$;
\item[]
$[Z,H_i] =0 $;
\item[]
$[R_i, R_j] = 0 $;
\item[]
$[R_i,A_{j,k}] = \left\{ \begin{array}{ll}
R_k , & \text{if } i=j,
\\
0 , & \text{if } i \neq j ;
\end{array} \right.$
\item[]
$[R_i,H_j] = \gamma_i^j R_i$;
\item[]
$[A_{ij},A_{k\ell}] = \left\{ \begin{array}{ll}
(v_i \otimes v_i)-(v_j \otimes v_j) =
\sum\limits_{r=1}^{n-2} \left( \gamma_i^r - \gamma_j^r \right) H_r , & \text{if } i=\ell,j=k,
\\
A_{i \ell} , & \text{if } i \neq \ell, j=k,
\\
-A_{kj} ,  & \text{if } i=\ell, j \neq k ,
\\
0,  & \text{if } i \neq \ell, j \neq k ;
\end{array} \right. $
\item[]
$[A_{ij}, H_k] = \left( \gamma_j^k-\gamma_i^k \right) A_{ij}$;
\item[]
$[H_i,H_j] = 0 . \ \square$
\end{itemize}
\end{lemma}

\begin{remark}
\label{Rmk properties of AH}
Lemma \ref{L multiplication table} shows that the orthogonal subspaces $\mathcal A$, $\mathcal H$ possess remarkable properties:
\begin{itemize}
\item[1.]
$\mathcal{H}$ is a Cartan subalgebra of $\mathfrak{l}$.

\item[2.]
$[\mathcal{H},\mathcal{A}] \subset \mathcal{A}$. The adjoint action of $\mathcal{H}$ on $\mathcal{A}$ is diagonal, for $H \in \mathcal H$:
\[
{\rm ad} H A_{ij}=(\gamma_{i}-\gamma_{j})A_{ij}.
\]

\item[3.]
$ [A_{ij}, A_{ji}]=v_i \otimes v_i - v_j \otimes v_j=H_{ij} \in \mathcal{H}$.

\item[4.]
For $i \neq j$, $\{A_{ij}, A_{ji}, [ A_{ij}, A_{ji} ] \}$ spans a $3$-dimensional Lie subalgebra:
\[
\left[ H_{ij}, A_{ij} \right] = 2A_{ij}, \quad
\left[ H_{ij}, A_{ji} \right] = -2A_{ji}.
\]

\item[5.]
For any $(ij),(k\ell)$ the commutator $[A_{ij}, A_{k\ell}]=\mbox{\rm ad}A_{ij} A_{k\ell}$ is orthogonal to
$A_{k\ell}$ with respect to the matrix scalar product.
$\square$
\end{itemize}
\end{remark}

\section{Semisimplicity of $\mathfrak l$}

In this section, we prove semisimplicity of $\mathfrak l$ by direct computation of the Killing form $\mathfrak B$.

\begin{proposition}
\label{P Killing}
The Killing form $\mathfrak B$ satisfies:
\begin{itemize}
\item[i)]
$\mathfrak{B}(\mathcal{A},\mathcal{H})=0$,
\item[ii)]
$\mathfrak{B}(H_i,H_j)=2(n-1)\langle H_i, H_j \rangle $, for $i,j=1, \ldots, n-2$,
\item[iii)]
$ \mathfrak B (A_{ij},A_{k\ell} ) = \left\{ \begin{array}{ll}
0, & \text{if } (i,j) \neq (\ell, k) ,
\smallskip \\
2(n-1), & \text{if } (i,j)= (\ell, k) . \ \square
\end{array} \right. $
\end{itemize}
\end{proposition}

According to Cartan criterion for semisimplicity, we get

\begin{corollary}
The Killing form $\mathfrak{B}$ is non-degenerate and the algebra ${\mathfrak l}$ is semisimple. $\square$
\end{corollary}

\begin{proof}[Proof of Proposition \ref{P Killing}]
{\it (i)}
Take $A_{ij}$, $H_k$ from the basis of $\mathcal A$ and $\mathcal H$, respectively.

Since $\mathcal H$ is Abelian, $\left( \mathrm{ad} A_{ij} \mathrm{ad} H_k \right)_{\mathcal H} =0$.

Due to Lemma \ref{L multiplication table}, for any $A_{\ell m}$, $\mathrm{ad} A_{ij} \mathrm{ad} H_k A_{\ell m} = C \mathrm{ad} A_{ij} A_{\ell m}$.
By property 5 in Remark \ref{Rmk properties of AH}, the last matrix is orthogonal to $A_{\ell m}$ and therefore  the trace of the restriction $\left.(\mathrm{ad} A_{i,j} \mathrm{ad} H_k)\right|_{\mathcal{A}}$ is null, and we can conclude that $\mathfrak{B}(\mathcal{A},\mathcal{H})=0$.
\smallskip

{\it (ii)}
Choose $H_k, H_\ell \in \mathcal H$.
 As far as $\left.(\mathrm{ad} H_k \mathrm{ad} H_\ell)\right|_{\mathcal{H}}=0$, we only need to  compute the trace of  $\left.(\mathrm{ad} H_k \mathrm{ad} H_\ell)\right|_{\mathcal{A}}$.

By Lemma \ref{L multiplication table}, $\mathrm{ad} H_k \mathrm{ad} H_\ell A_{ij } = \mathrm{ad} H_k (\gamma_i^\ell - \gamma_j^\ell )A_{ij} = (\gamma_i^\ell - \gamma_j^\ell)(\gamma_i^k - \gamma_j^k) A_{ij}$.
Hence,
\begin{align*}
&
\mathfrak B( H_k, H_\ell) = 
\sum_{i,j}(\gamma_i^\ell - \gamma_j^\ell)(\gamma_i^k - \gamma_j^k) =
\\ = &
(n-1)\sum_i \gamma_i^\ell\gamma_i^k -
\sum_i \gamma_i^\ell \sum_j \gamma_j^k -
\sum_j \gamma_j^\ell \sum_i \gamma_i^k +
(n-1) \sum_j \gamma_j^\ell \gamma_j^k .
\end{align*}
Since $\sum\limits_i \gamma_i^k =0$, it follows that
\[
\mathfrak B( H_k, H_\ell) = 2 (n-1) \sum\limits_i \gamma_i^\ell\gamma_i^k = 2(n-1) \langle H_k , H_\ell \rangle .
\]

{\it (iii)}
Pick  $A_{ij},A_{k\ell} $. For every    $H_m $  
\begin{equation}
\label{AAh}
\mathrm{ad} A_{ij}\mathrm{ad} A_{k\ell} H_m =
\mathrm{ad} A_{ij} (\gamma_\ell^m-\gamma_k^m) A_{k\ell} ,
\end{equation}
lies in $\mathcal A$ whenever $(k,\ell) \neq (j,i)$.
This implies
\[
{\rm Tr}\left.(\mathrm{ad} A_{ij} \mathrm{ad} A_{k\ell})\right|_{\mathcal{H}}=0,  \ \mbox{for} \ (k,\ell) \neq (j,i).
\]
To compute ${\rm Tr}\left.(\mathrm{ad} A_{ij} \mathrm{ad} A_{k\ell})\right|_{\mathcal{A}}$, notice that
\begin{align*}
&
\langle A_{\alpha\beta}, \mathrm{ad} A_{ij} \mathrm{ad} A_{k\ell} A_{\alpha\beta} \rangle =
v_\alpha^* \left( A_{ij} \mathrm{ad} A_{k\ell} A_{\alpha\beta} - (\mathrm{ad} A_{k\ell} A_{\alpha\beta})A_{ij} \right) v_\beta =
\\ = &
(v_\alpha \cdot v_i) v_j^* \left( A_{k\ell}A_{\alpha\beta} - A_{\alpha\beta}A_{ij} \right) v_\beta -
(v_\beta \cdot v_j) v_\alpha^* \left(A_{k\ell}A_{\alpha\beta}- A_{\alpha\beta}A_{k\ell} \right) v_i.
\end{align*}
Since $i \neq j$ and $k \neq \ell$, $v_j^*  A_{\alpha\beta}A_{ij} v_\beta = v_\alpha^* A_{k\ell}A_{\alpha\beta} v_i =0$, and therefore
\begin{align}
\notag &
\langle A_{\alpha\beta}, \mathrm{ad} A_{ij} \mathrm{ad} A_{k\ell} A_{\alpha\beta} \rangle =
\\ = &  \label{AAAA}
(v_j \cdot v_k)(v_i \cdot v_\alpha)(v_\ell \cdot v_\alpha) +
(v_i \cdot v_\ell)(v_j \cdot v_\beta)(v_k \cdot v_\beta ),
\end{align}
which is zero whenever $(k,\ell) \neq (j,i)$.

For $(k,\ell) = (j,i)$, the equality
\eqref{AAh} and Lemma \ref{L multiplication table} yield
\begin{align*}
\langle H_m , \mathrm{ad} A_{ij} \mathrm{ad} A_{ji} H_m \rangle = &
(\gamma_i^m-\gamma_j^m) \langle H_m , \mathrm{ad} A_{ij}A_{ji} \rangle =
\\ = &
(\gamma_i^m-\gamma_j^m) \langle H_m , v_i \otimes v_i - v_j \otimes v_j \rangle =
(\gamma_i^m-\gamma_j^m) ^2 ,
\end{align*}
and ${\rm Tr}\left.\left( \mathrm{ad} A_{ij} \mathrm{ad} A_{ji} \right)\right|_{\mathcal H} =
\sum\limits_{m=1}^{n-2}(\gamma_i^m - \gamma_j^m)^2$.

To compute the last expression, let us form the matrix
\begin{equation}
\label{Eq matrix Gamma}
\Gamma = \left( \begin{array}{ccc}
\gamma_1^1 & \cdots & \gamma_1^{n-2} \\
\vdots & & \vdots \\
\gamma_{n-1}^1 & \cdots & \gamma_{n-1}^{n-2}
\end{array}\right) .
\end{equation}
Then
$\Gamma \Gamma^*$ is the matrix of the orthogonal projection of $\mathbb R^{n-1}$ onto the subspace $\Pi_{n-1}$.
Take a standard basis $e_1, \ldots , e_{n-1}$ in $\mathbb{R}^{n-1}$, and note that $e_i-e_j \in \Pi_{n-1}$.
Then
\begin{eqnarray*}
{\rm Tr}\left.\left( \mathrm{ad} A_{ij} \mathrm{ad} A_{ji} \right)\right|_{\mathcal H} &=& \sum_{m=1}^{n-2}(\gamma_i^m - \gamma_j^m)^2 = \\
 =(e_i - e_j)^* \Gamma \Gamma^* (e_i-e_j)& =&
(e_i-e_j)^* (e_i-e_j) = 2.
\end{eqnarray*}

In what regards ${\rm Tr}\left.\left( \mathrm{ad} A_{ij} \mathrm{ad} A_{ji} \right)\right|_{\mathcal A} $, then by \eqref{AAAA}:
\[
\langle A_{\alpha\beta} , \mathrm{ad} A_{ij} \mathrm{ad} A_{ji} A_{\alpha\beta} \rangle =
(v_i \cdot v_\alpha ) + (v_j \cdot v_\beta ) .
\]
Hence,
\begin{align*}
{\rm Tr}\left.\left( \mathrm{ad} A_{ij} \mathrm{ad} A_{ji} \right)\right|_{\mathcal A} = &
\sum_{\begin{array}{c}
\alpha, \beta \leq n-1 \\ \alpha \neq \beta
\end{array}
} \left( (v_i \cdot v_\alpha ) + (v_j \cdot v_\beta ) \right) = 2(n-2) ,
\end{align*}
and therefore ${\rm Tr} \left( \mathrm{ad} A_{ij} \mathrm{ad} A_{ji} \right) = 2(n-1)$.
\end{proof}

\section{Classification of the Levi subalgebra $\mathfrak l $}

Now we wish to prove the following result concerning the type of the semisimple subalgebra $\mathfrak l $.

\begin{theorem}
\label{T classification}
The Levi subalgebra $\mathfrak l$ is isomorphic to the special linear Lie algebra $\mathfrak{sl}(n-1, \mathbb R)$.
$\square$
\end{theorem}

\begin{proof}
As stated in Remark \ref{Rmk properties of AH}, $\mathcal H$ is a Cartan subalgebra of $\mathfrak l$.
From Lemma \ref{L multiplication table}, we see that the nonzero characteristic functions of $\mathfrak l$ with respect to $\mathcal H$ are the linear functionals $\alpha_{ij}: \mathcal H \mapsto \mathbb R$ such that
\begin{equation}
\notag
\alpha_{ij} (H_k) = \gamma^k_i-\gamma^k_j , \qquad \text{for } 1 \leq k \leq n-2, \ \ 1 \leq i , j \leq n-1, \ \ i \neq j ,
\end{equation}
and the corresponding characteristic spaces are
\begin{equation}
\notag
\mathcal A_{ij} = \left\{ t A_{ij}: t \in \mathbb R \right\} \qquad 1 \leq i , j \leq n-1, \ \ i \neq j .
\end{equation}
Thus, $\mathfrak l$ is split as
\begin{equation}
\notag
\mathfrak l = \mathcal H \oplus \bigoplus_{i \neq j} \mathcal A_{ij} .
\end{equation}
 Hence the set, $\mathcal R = \left\{ \alpha_{ij}: 1 \leq i \leq n-1, \ 1 \leq j \leq n-1, \ i \neq j\right\}$ is a root system of $\mathfrak l$.

Since the Killing form restricted to $\mathcal H$ is diagonal, the dual space $\mathcal H^*$ is provided with the inner product uniquely defined by
\begin{equation}
\notag
\left\langle \alpha_{ij}, \alpha_{\ell,m} \right\rangle =
\sum_{k=1}^{n-2} \left( \gamma_i^k - \gamma_j^k \right)\left( \gamma_\ell^k - \gamma_m^k \right) =
(e_i-e_j)^*(e_\ell - e_m)
\end{equation}
for every $\alpha_{ij},\alpha_{\ell m} \in \mathcal R$.
Thus, $\mathcal R$ is isomorphic to the root system
\begin{equation}
\notag
\mathcal E = \left\{ e_i-e_j : 1 \leq i \leq n-1, \ 1 \leq j \leq n-1, \ i \neq j \right\}
\end{equation}
on the hyperplane $\Pi_{n-1}$.
Since
\begin{equation}
\notag
e_\ell - e_m = \left\{ \begin{array}{ll}
\sum\limits_{i=\ell}^{m-1} \left( e_i - e_{i+1} \right), & \text{if } \ell < m ,
\\
\sum\limits_{i=m}^{\ell -1} -\left( e_i - e_{i+1} \right), & \text{if } \ell > m ,
\end{array} \right.
\end{equation}
it follows that the set
$
\Delta = \left\{ \alpha_{12}, \alpha_{23} ,\alpha_{34}, \ldots , \alpha_{(n-2)(n-1)} \right\}
$
is a system of positive simple roots.
Further,
\begin{align*}
&
\left\langle \alpha_{i(i+1)} , \alpha_{i(i+1)} \right\rangle = 2 \qquad 1 \leq i \leq n-2 ,
\\ &
2 \frac{\left\langle \alpha_{i(i+1)},\alpha_{j(j+1)} \right\rangle}{\left\langle \alpha_{i(i+1)},\alpha_{i(i+1)} \right\rangle} =
\left\{ \begin{array}{ll}
-1 & \text{if } |i-j|=1,
\\
0 & \text{if } |i-j|>1 .
\end{array} \right.
\end{align*}
Thus, the Dynkin diagram of $\mathfrak l$ is of type $A_{n-2}$, and therefore, $\mathfrak l$ is isomorphic to $\mathfrak{sl}(n-1, \mathbb R)$ (see, e.g., \cite[Chapter 14]{SagleWalde73}).
\begin{figure}[!h]
\begin{picture}(270,30)
\put(10,20){\circle{4}}
\put(8,10){$\alpha_{12}$}
\put(12,20){\line(1,0){50}}
\put(64,20){\circle{4}}
\put(62,10){$\alpha_{23}$}
\put(66,20){\line(1,0){50}}
\put(118,20){\circle{4}}
\put(116,10){$\alpha_{34}$}
\put(120,20){\line(1,0){30}}
\put(152,20){\line(1,0){2}}
\put(156,20){\line(1,0){2}}
\put(160,20){\line(1,0){2}}
\put(164,20){\line(1,0){2}}
\put(168,20){\line(1,0){30}}
\put(200,20){\circle{4}}
\put(198,10){$\alpha_{(n-3)(n-2)}$}
\put(202,20){\line(1,0){50}}
\put(254,20){\circle{4}}
\put(252,10){$\alpha_{(n-2)(n-1)}$}
\end{picture}
\caption{Dynkin diagram of $\mathfrak l$}
\end{figure}
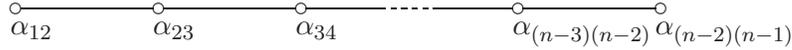
\end{proof}

\section{Representation of the Levi factor ${\mathfrak l}$ in $V=\mathbb{R}^{n}$}
\label{S representation l}

Considering ${\mathfrak l}$ as a subalgebra of the stochastic (matrix) algebra $\mathfrak{s}(n,\mathbb{R})$ defines its representation
$\phi: {\mathfrak l} \mapsto \mathfrak{gl}(n)$ in $V=\mathbb{R}^{n}$.
To characterize it, let us pick the basis $v_0, v_1, \ldots , v_{n-1}$, introduced in Section \ref{S Basis}, and consider the matrix $M\in \mathbb R^{n \times n}$: $M = \left( \begin{array}{c|c|c|c}
v_0 & v_1 & \cdots &v_{n-1}
\end{array} \right)$.

By construction, $M$ is orthogonal and  the mapping
\[\forall y \in \mathfrak l: \hspace*{5mm} y \mapsto M^*\phi(y)M,\]
defines an isomorphic representation of $\mathfrak l$ in $V=\mathbb R^n$.

Note that the subspace $V_1=\mbox{span}\{v_1, \ldots , v_{n-1}\}$ is invariant under $\phi(\mathfrak l)$
and therefore  we get:
\begin{equation}\label{rep_Rn}
\forall y \in \mathfrak l: \hspace*{5mm}  M^* \phi(y) M =
\left( \begin{array}{cc}
0 & 0 \\ 0 & M_1^*\phi(y)M_1
\end{array} \right),
\end{equation}
where $M_1 = \left( \begin{array}{c|c|c|c}
v_1 & v_2 & \cdots &v_{n-1}
\end{array} \right) \in \mathbb R^{n \times (n-1)}$.

The mapping \[y \mapsto \phi_1(y)=M_1^* \phi(y) M_1\] is a faithful representation of $\mathfrak l$ in $V_1=\mathbb R^{n-1}$.

Formula (\ref{rep_Rn}) identifies  the representation  of the semisimple Levi factor $\mathfrak l$  in $\mathbb{R}^n$ by stochastic matrices with  a direct sum of the
faithful representation $\phi_1$ in $\mathbb{R}^{n-1}$ and   the null $1$-dimensional  representation.

Besides
\begin{align*}
&
M_1^* A_{ij} M_1 = e_i \otimes e_j \qquad \text{for } i,j \in \{1,2,\ldots , n-1\}, \ i \neq j,
\\ &
M_1^* H_i M_1 = \mathrm{diag}(\gamma^i) \qquad \text{for } i = 1,2, \ldots, n-2.
\end{align*}
Therefore $\phi_1$ maps isomorphically  the Cartan subalgebra $\mathcal H$ onto the space of traceless diagonal $(n-1) \times (n-1)$ matrices, while $\phi_1( \mathcal A )$ coincides with the  space of  $(n-1) \times (n-1)$ matrices with vanishing diagonal.

\section{Affine group and affine Lie algebra}

It is noticed in \cite{Poole}  that the group of $\mathcal S(n,\mathbb{R})$ is isomorphic to the group ${\it Aff}(n-1,\mathbb R)$ of the affine maps $S: x \to Ax+B, \ x \in \mathbb{R}^{n-1}$.
We wish to discuss this relation, in the light of the   results obtained above.
We also discuss the relation between the elements of $\mathcal S (n, \mathbb R)$ and finite state space Markov processes outlined in Section \ref{S introduction}.

Let $\left( \mathbb R^n \right)^* $ be the dual of $\mathbb R^n$.
As usual, elements of $\mathbb R^n$ are identified with column vectors, and elements of $\left( \mathbb R^n \right)^*$ are identified with row vectors.
Further, we identify any vector $x=(x_1,x_2, \ldots, x_n) \in \mathbb R^n$ with the function $x: i \mapsto x_i$, with the domain $D_n= \{1,2, \ldots , n\}$, and identify any dual vector $p=(p_1,p_2, \ldots , p_n) \in \left( \mathbb R^n \right)^*$ with the (signed) measure $p$ on the set $D_n$ such that $p\{i\} =p_i$, for $i=1,2, \ldots, n$.
Thus, the product $px$ is identified with the integral $\int_{D_n} x \, dp$.

Each $S \in \mathcal S(n,\mathbb R) $ can be identified either with the linear endomorphism of $\mathbb R^n$, $S: x \mapsto Sx$
or with the linear endomorphism of $\left( \mathbb R^n \right)^*$, $S: p \mapsto pS$.

Let $Y$ be a $D_n$-valued Markov process and  $S \in \mathcal S^+(n,\mathbb R)$ be defined by  $
s_{ij} = \Pr \left\{ \left. Y_t=j \right| Y_s=i \right\} $ for every $ i, j \in D_n$ ($0\leq s \leq t < +\infty$, fixed). Then the vector $Sx$ is identified with the function $i \mapsto \mathbb E \left[ \left. x(Y_t) \right| Y_s=i \right]$, while the covector $pS$ is identified with the probability law of $Y_t$ assuming the probability law of $Y_s$ is $p$.

For every $S \in \mathcal S(n, \mathbb R)$, the map $p \mapsto pS$ preserves each affine space of the form  $\left\{ p \in \left( \mathbb R^n \right)^*: p \mathbf 1 = C \right\}$ ($C \in \mathbb R$, fixed), which is the space of  signed measures on $D_n$ such that $p(D_n) = C$.
Note that, $\{ t \mathbf 1 : t \in \mathbb R \}$ is the unique affine (linear) proper subspace of $\mathbb R^n$ which is preserved by all the maps $x \mapsto Sx$ with $S \in \mathcal S(n, \mathbb R)$.

Now, consider the group of  invertible affine maps $S:q \mapsto q A + B$, $q \in \left( \mathbb R^{n-1} \right)^*$.\footnote{The mapping
\begin{equation}
\notag
q = (q_1,q_2, \ldots ,q_{n-1} ) \mapsto \left( C- \sum\limits_{i=1}^{n-1} q_i, q_1, q_2, \ldots , q_{n-1} \right)
\end{equation}
coordinatizes the affine subspace $\left\{ p \in \left( \mathbb R^n \right)^*: p \mathbf 1 = C \right\}$.}
The group  can be identified with the subgroup $\mathbf A \left( \left( \mathbb R^{n-1} \right)^*\right)$ of $GL\left( \left( \mathbb R^n \right)^*\right)$:
\begin{equation}
\notag
\mathbf A \left( \left( \mathbb R^{n-1} \right)^*\right)=\left\{\left.\left( \begin{array}{cc}
1 & B \\ 0 & A
\end{array} \right) \right| \ A \in \mathbb R^{(n-1) \times (n-1)} \ \mbox{is nonsingular}\right\}.\end{equation}

The Lie algebra $\mathfrak a \left( \left( \mathbb R^{n-1} \right)^*\right)$ of $\mathbf A \left( \left( \mathbb R^{n-1} \right)^*\right)$ consists of matrices
\[\left( \begin{array}{cc}
0 & B \\ 0 & A
\end{array} \right).\]

Now, fix $S \in \mathcal S(n, \mathbb R)$.
By the results of Section \ref{S Basis}, $S$ can be written as 
\begin{equation}
\notag
S = \beta_0 Z + \sum_{i=1}^{n-1} \beta_i R_i + A,
\end{equation}
with $\beta_0, \beta_1, \ldots, \beta_{n-1} \in \mathbb R$, $A \in \mathfrak l$.
Taking into account that
\begin{align*}
&
Z v_0 = 0,
\\ &
Z v_i = \frac{1}{\sqrt{n-1}} v_i, \quad
R_i v_j =
\left\{ \begin{array}{ll}
v_0, & \text{if } j=i,
\\
0, & \text{if } j \neq i,
\end{array} \right.
\qquad
\text{for } i=1,2, \ldots, n-1,
\end{align*}
we get
\begin{equation}
\notag
M^* S M =
\left( \begin{array}{cc}
0 & \beta^* \\ 0 & M_1^*AM_1 + \frac{\beta_0}{\sqrt{n-1}}Id
\end{array} \right) ,
\end{equation}
where $\beta^* = (\beta_1, \beta_2, \ldots, \beta_{n-1} )$.
Thus, the similarity $S \mapsto M^* SM$ is an isomorphism from $\mathcal S(n, \mathbb R)$ into $\mathfrak a \left( \left( \mathbb R^{n-1} \right)^* \right)$.
In particular, the radical of $\mathfrak a \left( \left( \mathbb R^{n-1} \right)^* \right)$ is the linear space of matrices
\begin{equation}
\notag
\left( \begin{array}{cc}
0 & \beta^* \\ 0 & \beta_0 Id
\end{array} \right) ,
\qquad \beta_0, \beta_1 , \ldots , \beta_{n-1} \in \mathbb R ,
\end{equation}
while the Levi subalgebra of $\mathfrak a \left( \left( \mathbb R^{n-1} \right)^* \right)$ consists of matrices
\begin{equation}
\notag
\left( \begin{array}{cc}
0 & 0 \\ 0 & A
\end{array} \right) ,
\qquad A \in \mathfrak{sl}(n-1,\mathbb R ).
\end{equation}
Thus, the Levi splitting of $\mathfrak a \left( \left( \mathbb R^{n-1} \right)^* \right)$ corresponds to two connected Lie subgroups of $\mathbf A \left( \left( \mathbb R^{n-1} \right)^* \right)$:
The subgroup generated by the translations and rescalings of $\left( \mathbb R^{n-1} \right)^*$, and the subgroup of orientation and volume preserving linear transformations in $\left( \mathbb R^{n-1} \right)^*$.

\section{Minimal number of generators of $\mathfrak{s}(n,\mathbb R)$}

Finally we prove 

\begin{theorem}
\label{T generators of zrs}
The Lie algebra $\mathfrak{s}(n,\mathbb R)$ is generated by two matrices.
$\square$
\end{theorem}

The argument in our proof is an adaptation of the argument used in \cite{Kuranishi51} to prove that every semisimple Lie algebra is generated by two elements.
We will use the following lemma:

\begin{lemma}
\label{L vector}
For every integer $n \geq 2$ there is a vector $\gamma \in \mathbb R^n$ such that
\begin{itemize}
\item[a)]
$\sum\limits_{i=1}^n \gamma_i=0$;
\item[b)]
$\gamma_i \neq 0 , \qquad i=1, \ldots , n$;
\item[c)]
$\gamma_i \neq \gamma_j, \qquad \forall i,j \in \{1, \ldots, n\}, i \neq j$;
\item[d)]
$\gamma_i-\gamma_j \neq \gamma_k - \gamma_\ell , \qquad \forall i,j,k,\ell \in \{1, \ldots, n\}, i \neq j, k \neq \ell, (i,j) \neq (k,\ell)$.
\end{itemize}
For every $\gamma$ satisfying (a)--(d) and every $\lambda \in \mathbb R\setminus \{0 \}$, $\lambda \gamma$ satisfies (a)--(d).
$\square$
\end{lemma}

\begin{proof}
For $n=2$, the Lemma holds with $\gamma =(1,-1)$.

Suppose that the Lemma holds for some $n \geq 2$, and fix $\gamma \in \mathbb R^n$ satisfying {\it (a)--(d)}.
Let
\[
\tilde \gamma = \left( \gamma_1, \ldots , \gamma_{n-1}, \gamma_n - \varepsilon, \varepsilon \right) .
\]
Since there are only finitely many values of $\varepsilon$ such that $\tilde \gamma$ fails at least one condition {\it (a)--(d)}, we see that the Lemma holds for $n+1$.

The last statement in the Lemma is obvious, since the equations in conditions {\it (a)--(d)} are homogeneous.
\end{proof}

\begin{proof}[Proof of Theorem \ref{T generators of zrs}]
Pick a vector $\gamma \in \mathbb R^{n-1}$ satisfying  conditions {\it (a)--(d)} of Lemma \ref{L vector}, let $\Gamma$ be the matrix \eqref{Eq matrix Gamma}, and $\beta= (\beta_1, \ldots ,\beta_{n-2}) = \gamma^T \Gamma$.
Let
$Z, R_i, A_{ij}, H_i$ be elements of our basis of $\mathfrak{s}(n , \mathbb R)$, and consider the matrices
\[
X = Z + \sum_{k=1}^{n-2}\beta_k H_k,
\qquad
Y = R_1 + \sum_{i \neq j} A_{ij} .
\]
Using the Lemma \ref{L multiplication table}, we obtain
\begin{align*}
\textrm{ad}XY = &
[Z,R_1] + \sum_{i \neq j}[Z,A_{ij}] + \sum_{k=1}^{n-2}\beta_k [H_k,R_1] + \sum_{k=1}^{n-2}\beta_k [H_k, A_{ij}] =
\\ = &
\frac{-1}{n-1}R_1 + 0 - \gamma_1 R_1 + \sum_{i \neq j}(\gamma_i-\gamma_j) A_{ij} =
\\ = &
-\left( \frac{1}{n-1} + \gamma_1 \right) R_1 + \sum_{i \neq j}(\gamma_i-\gamma_j) A_{ij} .
\end{align*}
Multiplying $\gamma$ by an appropriate non zero constant we can make $\gamma_1=\frac{-1}{n-1}$, and thus
\[
\textrm{ad}XY = \sum_{i \neq j} (\gamma_i-\gamma_j) A_{ij} .
\]
Iterating, we see that
\[
\textrm{ad}^kXY = \sum_{i \neq j} (\gamma_i-\gamma_j)^k A_{ij} \qquad \forall k \in \mathbb N.
\]
Let $m = (n-1)(n-2)$.
Since
\[
\mathrm{det}\left( \begin{array}{ccccc}
1 & 0 & 0 & \cdots & 0 \\
0 & 1 & 1 & \cdots & 1 \\
0 & 0 & \gamma_1 - \gamma_2 & \cdots & \gamma_{n-1} - \gamma_{n-2} \\
0 & 0 & (\gamma_1 - \gamma_2)^2 & \cdots & (\gamma_{n-1} - \gamma_{n-2})^2 \\
\vdots & \vdots & \vdots & & \vdots \\
0 & 0 & (\gamma_1 - \gamma_2)^m & \cdots & (\gamma_{n-1} - \gamma_{n-2})^m \\
\end{array} \right)
\neq
0 ,
\]
we see that the matrices $X,Y, \textrm{ad}XY, \ldots , \textrm{ad}^mXY$ span the same subspace as the matrices $X,R_1, A_{ij}, \ i,j \leq n-1, i \neq j$, and this subspace lies in $\mathfrak{Lie}\{X,Y\}$, the Lie algebra generated by $X,Y$.

By the Lemma \ref{L multiplication table}, $[R_1,A_{1i}] = R_i$, for $i=1,2, \ldots , n-1$.
Hence \[\{ R_2,\ldots , R_{n-1} \} \subset \mathfrak{Lie}\{X,Y\}.\]

Finally, also by the Lemma \ref{L multiplication table}, $[A_{ij},A_{ji}] = \sum\limits_{r=1}^{n-2}(\gamma_i^r- \gamma_j^r) H_r $.
This implies that $[A_{1j},A_{j1} ], \ j=2, \ldots , n-1$ are $n-2$ linearly independent elements of $\mathcal H$.
Hence, $\mathcal H \subset \mathfrak{Lie}\{ X,Y\}$ and $Z \in \mathfrak{Lie}\{X,Y\}$.
\end{proof}

\section{Acknowledgement}
The authors are grateful to A.A.Agrachev for stimulating remarks.

\bibliographystyle{amsplain}

\end{document}